\documentclass[12pt]{article}
\usepackage{amsfonts,amsmath}
\usepackage{amssymb}
\usepackage{latexsym}
\usepackage{epsfig}
\usepackage{graphicx,color,graphics}

\usepackage{caption}
\usepackage{subcaption}
\usepackage{xcolor}
\usepackage{tcolorbox,color}
\usepackage{lineno}
\usepackage{pgf,tikz}
\usepackage{mathrsfs}
\usepackage{multicol}
\usepackage{enumitem}
\usepackage{mdwlist}
\usepackage{dsfont}
\usetikzlibrary{arrows}
\usetikzlibrary{patterns}
\newtheorem{lemma}{Lemma}[section]

\newtheorem{theorem}[lemma]{Theorem}

\newtheorem{remark}[lemma]{Remark}

\newcommand{\N}{\ifmmode{{\Bbb N}}\else{\mbox{${\Bbb N}$}}\fi}
\newcommand{\R}{\ifmmode{{\Bbb R}}\else{\mbox{${\Bbb R}$}}\fi}

\begin{document}
		\title{Thermoelastic plates with type I heat
			conduction with second gradient.}
\author{Jaime Mu\~noz Rivera$^{1,}$\footnote{jemunozrivera@gmail.com},\ \ Elena Ochoa Ochoa $^{2,}$\footnote{elenaochoaochoa18@gmail.com}, \ \  Ram\'on Quintanilla $^{3,}$ \footnote{ramon.quintanilla@upc.edu}.  \\  \small \it $^{1}$Universidad del B\'io-B\'io, Departamento de Matem\'atica, Facultad de Ciencias,\\
	\small \it Avenida Collao 1202, Concepci\'on, Chile.\\
	\small \it
	\textcolor{red}{ $^{1}$National Laboratory for Scientific Computation, Petropolis, Brazil}\\
	\small \it
	 $^{2}$Universidad Andres Bello, Departamento de Matemáticas, Facultad de Ciencias Exactas\\
	\small \it Sede Concepción, Autopista Concepción-Talcahuano 7100, Talcahuano, Chile.\\
	\small \it $^{3}$ Universidad Politécnica de Catalunya,  Departamento de Matemáticas,\\
	\small \it Colom 11, 08222 Terrassa, Spain.			    
}		
\date{}
	\maketitle

	\begin{abstract}

This paper investigates the qualitative properties of thermoelastic plates
modeled by the second-gradient theory with a Type I heat equation. We establish
the exponential stability of the solutions. Our main contribution is to prove that the
semigroup is non-differentiable when the bi-Laplacian operator appears in the heat
equation. Additionally, we analyze the case where the elastic parameter is negative, demonstrating the uniqueness and instability of the solutions. Finally, in the one-dimensional quasi-static case, we demonstrate the existence and exponential decay of the solutions under specific conditions.

\end{abstract}


		\section{Introduction}


	Extensive research has been conducted on the asymptotic behavior of thermoelastic plates in bounded domains. When the mechanical component is conservative and the thermal dissipation is parabolic, the solutions are guaranteed to be exponentially stable and semigroup analytic, as demonstrated by the referenced sources \cite{Z1,Z2,Z3}.
	 However, such regular behavior cannot be expected in the case of the Lord-Shulman type of dissipation, in which both exponential stability and analyticity are lost \cite{R1}. Similarly, in the Green-Lindsay case, analyticity is absent, though exponential stability can still be ensured \cite{R2}. Recent studies have also investigated plates with mechanical dissipation and thermal conservation, yielding different results \cite{M2,M3}. When the plate occupies the entire space, alternative results have been obtained \cite{R2}.\\	
	It is commonly assumed that dissipation induces a regularizing effect, suggesting that coupling with a regularity property in the solutions will be preserved or even enhanced with increased dissipation.  However, recent studies have challenged this intuition by revealing cases in which it does not align with the mathematical analysis. Notably, introducing the bi-Laplacian operator into certain couplings has been shown to eliminate the exponential stability that would otherwise be present \cite{M1}. Furthermore, these effects have only been examined in the context of second-order equations with respect to the time variable.

Recently, there has been a growing interest in studying the impact of higher-order spatial derivatives on various thermomechanical issues. These equations are well-established in the context of the elastic component and have been studied for several years. However, there has also been an increased focus on incorporating these terms into equations that governing porosity and heat transfer. The interest in heat equations may be driven by phenomena observed in various gases. Furthermore, including gradient effects in thermomechanical models has been shown to be significant \cite{M2.1}. The work of Iesan \cite{D1,D2,D3,D4, D5} is particularly noteworthy in this regard. Building on the framework of Green and Naghdi formulation \cite{[5],[6],[7]}, Iesan has developed distinct theories  that incorporate second-gradient effects into the heat equation.

	This article focuses on analyzing the qualitative properties of thermoelastic plates under the second-gradient theory of the Type I heat equation. This theory, is derived from the framework proposed in \cite{D2}, and can be is obtained by restricting the considered independent variables, since Type I theories are known to be a sub-class of Type III theories. The system of equations that governs this problem is expressed as follows:
		\begin{align}
	\rho u_{tt}&=-c\Delta^2 u+\eta \Delta \theta,\label{eq1}\\
		a \theta_t&= b \Delta \theta - d\Delta^2 \theta- \eta \Delta u_t. \label{eq2}
	\end{align}
	In this system, $u$ represents the displacement, $\theta$  denotes the temperature, $\rho$  is the mass density, $a$ is the heat capacity, $c$ is the  elasticity coefficient, $b$ is the thermal conductivity, $\eta$  is the coupling constant, and $d$ is a novel parameter introduced in the higher-order theory. All parameters are assumed to be constant, with 
$\rho$, $a$, $b$, $c$ and $d$ taken as positive, while $\eta$ is only required to be nonzero, except in the final setting, where we assume $c<0$.

We study the problem governed by this system in a bounded domain $\Omega$  with a smooth boundary in an $n$-dimensional Euclidean space. To fully define the problem, we must specify the appropriate boundary conditions:
		\begin{align}
u({\bf {x}}, t)=\Delta u({\bf {x}}, t)=\theta ({\bf {x}}, t)=\Delta \theta ({\bf {x}}, t)=0,\, \, \, {\bf {x}}\in \partial \Omega , \, \, t>0, \label{fr1}
	\end{align}

	as well as some initial conditions:
		\begin{align}
	u({\bf {x}}, 0)=u^0({\bf {x}}),\, \, \, u_t({\bf {x}}, 0)=v^0({\bf {x}}),\, \, \,\theta ({\bf {x}}, 0)=\theta^0({\bf {x}}).\label{ci}
	\end{align}
	
	In this article, we prove  the well-posedness of the proposed system, and its initial and boundary conditions, in the Hadamard sense. We prove the existence
	of a contraction semigroup that governing the solutions, and we demonstrate their exponential 	
	stability. The primary novel contribution of this work is the proof that the semigroup is
	non-differentiable, and thus non-analytic, despite the analyticity of the semigroup when
	$d = 0$, as established in \cite{Z1}. This loss of regularity due to additional dissipation is
	unexpected.\footnote{Notably, this is the first instance of such a phenomenon observed in  a first-order equation in the time variable of the heat equation.} Nevertheless, we show that the only solution that vanishes on a set of positive
	measure is the trivial solution, thereby confirming the uniqueness of solutions for the
	backward-in-time problem.\\
	The manuscript is organized as follows: Section $2$ establishes the existence and
	uniqueness of solutions for the model. Section 3 proves exponential stability. Section
	4 demonstrates the non-differentiability of the semigroup. Section $5$ addresses the
	impossibility of localization. In Sections $6$ and $7$, we analyze the case where $c < 0$, confirming the uniqueness and instability of solutions. It is relevant recalling that the elasticity coefficient can be negative in the case for  pre-stressed materials (see \cite{D6}). Therefore, this assumption is compatible with the usual axioms of thermomechanics. However, for the one-dimensional
	quasi-static case, we prove the existence and exponential decay of solutions under
	specific conditions.


	
To investigate the dissipative properties, we construct the energy functionals.
 To this end, we multiply the equation \eqref{eq1} by  $u_t$  and equation \eqref{eq2} by $\theta$, yielding the following:
		\begin{align*}
			\dfrac{1}{2}\dfrac{d}{dt} \int_{\Omega} \left( \rho \left| u_t\right|^2+c \left| \Delta u\right|^2 \   \right) d\Omega =-\eta \int_{\Omega} \nabla \theta \nabla u_{t} \ d\Omega,\\
				\dfrac{1}{2}\dfrac{d}{dt} \int_{\Omega} a \left| \theta\right|^2 \ d\Omega + \int_{B} \left(b \left| \nabla \theta \right|^2 + d \left| \Delta \theta  \right| ^2  \right) \ d\Omega= \eta \int_{\Omega} \nabla \theta \nabla u_t \ d\Omega.
		\end{align*}
		Thus we see that,
			\begin{align*}
		\dfrac{1}{2}\dfrac{d}{dt} \int_{\Omega} \left(  \rho \left| u_t\right|^2+c \left| \Delta u\right|^2 + a \left| \theta\right|^2\right) \ d\Omega = -\int_{\Omega}\left(b \left| \nabla \theta \right|^2 + d \left| \Delta \theta  \right| ^2  \right) \ d\Omega.
	\end{align*}
In short we can write
\begin{align}
	E(t)+\int_{0}^{t} D(s) \ ds = E(0), \label{eq1.5}
\end{align}
where
\begin{align*}
	E(t)=\dfrac{1}{2} \int_{\Omega} \left(  \rho \left| u_t\right|^2+c \left| \Delta u\right|^2 + a \left| \theta\right|^2\right) \ d\Omega,
\end{align*}	 	
and
\begin{align*}
D(t)= \int_{\Omega}\left(b \left| \nabla \theta \right|^2 + d \left| \Delta \theta  \right| ^2  \right) \ d\Omega.
\end{align*}	 	

	\section{Existence and Uniqueness. } 	
	
	We are going to consider our problem \eqref{eq1}--\eqref{ci} in the Hilbert space
	$$\mathcal{H}=H_0^1(\Omega)\cap H^2(\Omega)\times L^2(\Omega) \times L^2 (\Omega), $$
	where $H_0^1$, $H^2$, and $L^2$ are the usual Sobolev spaces. The elements of this space can be written  as $U=(u,v,\theta)$. We can consider the inner product associated with the norm
	$$\left\|U \right\|_{\mathcal{H}} ^2= c\left\| \Delta u \right\|^2 +\rho \left\|v \right\|^2+a \left\|\theta \right\|^2.$$
	Note that this norm is equivalent to the usual one in $\mathcal{H}$. Therefore, we can write our problem in the form
	\begin{align}\label{Cauchy}
		U_t=\mathcal{A} U, \ \ U(0)=(u^0, v^0, \theta^0),
	\end{align}
	where the operator $\mathcal{A}$ is defined by
	\begin{align}\label{DA1}
		\mathcal{A}U=
\begin{pmatrix}
	v\\
	\dfrac{1}{\rho}\left(-c\Delta ^2 u+\eta \Delta \theta \right)  \\
\dfrac{	1}{a} \left( b\Delta \theta -d\Delta ^2 \theta-\eta \Delta v\right) 
\end{pmatrix}.
	\end{align}
	Here we consider the domain of the operator 
	\begin{align}\label{DA2}
	D(\mathcal{A})=H_*^4(\Omega)\times [H^2(\Omega)\cap H_0^1(\Omega)]\times H_*^4(\Omega) ,
	\end{align}
	where 
	$$
	H_*^4(\Omega)=\left\{u\in H^4(\Omega); \Delta u=u=0 \  \text{at the boundary} \right\}.
	$$
	It is not difficult to see that $	D(\mathcal{A})$  is a dense subspace of $	\mathcal{H}$. Moreover the operator $\mathcal{A}$ is dissipative, that is 
	\begin{align}
	\mbox{Re }\left\langle\mathcal{A}U,U \right\rangle=-\int_{\Omega}(b\left|\nabla \theta \right|^2+d\left| \Delta \theta \right|^2  ) \ d\Omega\leq 0,\quad \forall U\in D(\mathcal{A}). \label{dis1}
	\end{align}
	
	Let us consider $U=(u,v,\theta)^\top \in D(\mathcal{A})$ and $G=(g_1,g_2,g_3)^\top \in \mathcal{H}. $ The resolvent equation $i\omega U-AU=G$ in the terms of its component can be written as
			\begin{align}
		i\omega u-v= g_1 \in H^2,\label{ResO1}\\
		i\rho \omega v+c\Delta^2 u-\eta \Delta \theta= g_2 \in L^2,\label{ResO2}\\
		ia\omega \theta-b\Delta \theta+d\Delta^2\theta+\eta\Delta v= g_3\in L^2.\label{ResO3}
	\end{align}
	From \eqref{dis1} and the resolvent equation $i\omega U-\mathcal{A}U=G$ we get
\begin{eqnarray}\label{Diss}
\int_{\Omega}b\left( \left|\nabla \theta \right|^2+d\left| \Delta \theta \right|^2 \right)  \ d\Omega&=&\mbox{Re }(U,G)_{\mathcal{H}}.
\end{eqnarray}

\begin{theorem}
The operator $\mathcal{A}$, defined by equations \eqref{DA1} and \eqref{DA2}, generates a $C^0$-semigroup of contractions on the Hilbert space $\mathcal{H}$.
\end{theorem}

\begin{proof}
To show that $\mathcal{A}$  is the infinitesimal generator of a contraction semigroup, it is sufficient to show that 0 belongs to the resolvent of the operator. Indeed, let us take $G=(g_1,g_2,g_3)\in\mathcal{H}$ we want to find $U=(u,v, \theta)$ such that $-\mathcal{A}U=G.$ In terms of the components ($\omega=0$) we have 
\begin{align*}
	v&=-g_1,\\
	\Delta(-c\Delta u+\eta \theta)&= -\rho g_2,\\
	\Delta \left(b\theta-d\Delta \theta-\eta v \right) 	&=-ag_3.
\end{align*}
By definition we get  $v\in H_0^1\cap H^2$. By substituting $v$ into the next two equations, we arrive at 
\begin{align}
		\Delta(-c\Delta u+\eta \theta)&= -\rho g_2,\label{Res1}\\
	\Delta \left(b\theta-d\Delta \theta \right) 	&=-ag_3- \eta \Delta g_1 .\label{Res2}
\end{align}
Using  the Lax-Milgram's Lemma it is easy to show that for any $g_3\in L^2(\Omega)$ and  any $g_1\in H^2(\Omega)\cap H_0^1(\Omega)$ there exists only one solution $\theta\in H^2(\Omega)\cap H_0^1(\Omega)$ which is a weak solution to \eqref{Res2}. 
Using equation \eqref{Res2} and the elliptic regularity we conclude that $\theta\in H_*^4(\Omega)$. 
Since equation \eqref{Res1} and equation \eqref{Res2} are decoupled we consider $\theta$ as a data in equation \eqref{Res1}.  So we have 
$$
c\Delta^2 u= -\eta\Delta\theta +\rho g_2\in L^2(\Omega).
$$
Using  the Lax-Milgram's Lemma and the elliptic regularity again, we conclude that  there exists only one solution $u\in H_*^4(\Omega)$. In summary we proved that for any $G \in \mathcal{H}$ there is only one solution $U=(u,v,\theta)\in D(\mathcal{A})$ verifying $\mathcal{A}U=G$  and also 

$$ \left\|U \right\|_{\mathcal{H}} \leq C \left\| G \right\|_{\mathcal{H}}, $$
hence zero belongs to the resolvent set $\varrho(\mathcal{A})$. 

\end{proof}
\bigskip
In particular we have: 

\begin{theorem}\label{existencia2} There is only one solution to the Cauchy problem    \eqref{Cauchy} which depends continuously on the initial data. Moreover 
if $U_0\in \mathcal{H}$, there exists only one solution $U$ verifying
$$U(t)\in C([0,\infty);\mathcal{H}). $$
if  $U_0\in D({\mathcal{A}})$, then the solution has the regularity:
$$U (t)\in C^{1}([0,\infty);\mathcal{H})\cap C([0,\infty);D({\mathcal{A}})). $$
\end{theorem}

	\section{Exponential Stability. } 	
	
In this section, we will prove that the solutions to the problems studied in the previous section decay exponentially.  Our main tool is the following result by J. Pruess.  \cite{Pr84}.

\begin{theorem}\label{Pruss}
Let $ S(t)$ be a contraction $C_{0}$-semigroup, generated by  ${\mathcal A}$
over a Hilbert space ${\cal H}$. Then,  there exists $C, \gamma>0$ verifying
\begin{equation}\label{ppp}
\|S(t)\|\leq Ce^{-\gamma t}\quad \Leftrightarrow\quad i\,\mathbb{R}\subset\varrho({\mathcal A}) \, \mbox{ and }\,
\|(i\,\lambda\,I - {\mathcal A})^{-1}\|_{{\cal L}({\cal H})}
\leqslant M, \;\; \forall\,\lambda\in\mathbb{R}.
\end{equation}
\end{theorem}

Our first step in proving exponential stability is to show that the imaginary axis is contained at the resolvent set $\varrho(\mathcal{A})$.

\begin{lemma}\label{iR}
The operator $\mathcal{A}$ defined by \eqref{DA1} and \eqref{DA2} 
verifies $i \mathbb{R}\subset \varrho(\mathcal{A})$.
\end{lemma}
\begin{proof} 
Since the domain $D(\mathcal{A})$ is compactly embedded in the phase space $\mathcal{H}$, it suffices to show that $\mathcal{A}$ has no imaginary eigenvalues. We proceed by contradiction. Suppose there exists a non-zero $U \neq 0$ such that $\mathcal{A}U = i\omega U$. Given $G = 0$, equation \eqref{Diss} implies $\Delta\theta = \Delta^2\theta = 0$. Combined with the boundary conditions, this yields $\theta = 0$. From equation \eqref{ResO3}, we obtain $\Delta v = 0$, and applying the boundary conditions, we conclude that $v = 0$, which further implies $u = 0$. Thus, $U = (u, v, \theta) = (0, 0, 0)$, contradicting the assumption that $U \neq 0$.
This completes the proof.
\end{proof}
\bigskip
	
	Under the above conditions we have: 
		 
	 \begin{theorem}
	 The semigroup $S(t)$ generated by the operator $\mathcal{A}$ is exponentially stable. That is, there exist two positive constants $M$ and $\varepsilon$ such that
	 \begin{align*}
	 	\left\| U(t)\right\| _{\mathcal{H}}\leq Me^{-\varepsilon t}\left\|U(0) \right\| _{\mathcal{H}}.
	 \end{align*}

	 \end{theorem}
\begin{proof}
We need to  prove that the resolvent operator is uniformly bounded over the imaginary axis. Using the  equations:
\eqref{ResO3} and \eqref{ResO1}  we get
$$
		ia\omega \theta-b\Delta \theta+d\Delta^2\theta+\eta i\omega\Delta u=\eta \Delta g_1+ g_3 \in L^2,
		$$
and dividing by $i\omega $ we see
$$
		\theta-\frac{b}{i\omega }\Delta \theta+\frac{d}{i\omega }\Delta^2\theta+\eta \Delta u=\frac{1}{i\omega }\left(\eta \Delta g_1+ g_3\right).
		$$
Multiplying the above equation by $\overline{\Delta u}$ we obtain

\begin{eqnarray}
\int_{\Omega}\eta|\Delta u|^2dx&=&-\int_{\Omega}\theta\overline{\Delta u}+\frac{b}{i\omega }\int_{\Omega}\Delta \theta\overline{\Delta u}-\int_{\Omega}\frac{d}{i\omega }\Delta^2\theta\overline{\Delta u}\nonumber\\
&&+\frac{1}{i\omega }\int_{\Omega}\left(\eta \Delta g_1+ g_3\right)\overline{\Delta u}\nonumber\\
&=&-\int_{\Omega}\theta\overline{\Delta u}+\frac{b}{i\omega }\int_{\Omega}\Delta \theta\overline{\Delta u}\underbrace{-\frac{d}{i\omega }\int_{\Omega}\Delta\theta\overline{\Delta^2 u}}_{:=J_1}\nonumber\\
&&+\underbrace{\frac{1}{i\omega }\int_{\Omega}\left(\eta \Delta g_1+ g_3\right)\overline{\Delta u}}_{\leq \frac{\widetilde{c}}{|\omega| }\|U\|_{\mathcal{H}}\|G\|_{\mathcal{H}}}\label{Du}.
\end{eqnarray}
Using equation 
\eqref{ResO2}  we get
\begin{eqnarray*}
J_1&=&-\frac{d}{i\omega }\int_{\Omega}\Delta\theta(\overline{i\rho \omega v-\eta \Delta \theta-g_2})d\Omega\\
&=&d\rho \int_{\Omega}\Delta\theta \,\overline{v}\,d\Omega +\frac{d}{i\omega }\int_{\Omega}\eta|\Delta\theta|^2d\Omega+\frac{d}{i\omega }\int_{\Omega}\Delta\theta\overline{g_2}d\Omega.
\end{eqnarray*}
So, we arrive to 
\begin{eqnarray*}
\left|J_1\right|&\leq &
\widetilde{c}_\epsilon \int_{\Omega}|\Delta\theta|^2\,d\Omega +\epsilon \int_{\Omega}|v|^2\,d\Omega +\widetilde{c}\|G\|_{\mathcal{H}}^2.
\end{eqnarray*}
Hence using \eqref{Diss}
\begin{eqnarray*}
\left|J_1\right|&\leq &
\widetilde{c}_\epsilon\|U\|_{\mathcal{H}}\|G\|_{\mathcal{H}} +\epsilon \int_{\Omega}|v|^2\,d\Omega +\widetilde{c}\|G\|_{\mathcal{H}}^2.
\end{eqnarray*}
Substitution into \eqref{Du} we get 
\begin{eqnarray}\label{DDU}
\int_{\Omega}|\Delta u|^2\;d\Omega
&\leq & \widetilde{c}_\epsilon\|U\|_{\mathcal{H}}\|G\|_{\mathcal{H}} +\epsilon \int_{\Omega}|v|^2\,d\Omega +\widetilde{c}\|G\|_{\mathcal{H}}^2 ,
\end{eqnarray}
where we have used the Poincaré inequality type 
$\|\theta\|\leq \widetilde{c}\|\Delta\theta\|$. Finally, multiplying equation \eqref{ResO2} by $\overline{u}$ we find
$$
		i\rho \int_{\Omega}\omega v\overline{u}\;d\Omega+c\int_{\Omega}|\Delta u |^2\;d\Omega-\eta \int_{\Omega}\Delta \theta \overline{u}\;d\Omega= \int_{\Omega}g_2  \overline{u}\;d\Omega.
		$$
Using equation \eqref{ResO1}  we get
$$
		 \rho \int_{\Omega}|v|^2\;d\Omega=-\rho \int_{\Omega} v\overline{g_1}\;d\Omega+c\int_{\Omega}|\Delta u |^2\;d\Omega-\eta \int_{\Omega}\Delta \theta \overline{u}\;d\Omega- \int_{\Omega}g_2  \overline{u}\;d\Omega.
		$$
The above inequality implies 
$$
  \int_{\Omega}|v|^2\;d\Omega\leq c\int_{\Omega}|\Delta u |^2\;d\Omega+c\int_{\Omega}|\Delta \theta |^2\;d\Omega+\widetilde{c}_\epsilon\|U\|_{\mathcal{H}}\|G\|_{\mathcal{H}}.
		$$
Using \eqref{DDU} and  \eqref{Diss} we get
$$
  \int_{\Omega}|v|^2\;d\Omega\leq \widetilde{c}_\epsilon\|U\|_{\mathcal{H}}\|G\|_{\mathcal{H}}+\widetilde{c}\|G\|_{\mathcal{H}}^2 ,
		$$
for $\epsilon$ small. 
The above inequality implies that 
$$
\|U\|_{\mathcal{H}}^2=  \int_{\Omega}|\left( v|^2+|\Delta u|^2+|\theta |^2\right) \;d\Omega\leq \widetilde{c}_\epsilon\|U\|_{\mathcal{H}}\|G\|_{\mathcal{H}}+\widetilde{c}\|G\|_{\mathcal{H}}^2. 
		$$
Therefore, it follows the existence of a positive constant $M$ such that $\|U\|_{\mathcal{H}}\leq M \|G\|_{\mathcal{H}}$ and the proof is complete.

\end{proof}

		\section{Lack of Differentiability. }

		In this section, we will show that the semigroup associated with system \eqref{eq1}--\eqref{ci} is not  differentiable \cite{Pazy}
  (not immediately differentiable \cite{EngelNagel} ). To see this, we recall the following results.

\begin{theorem} \label{diff}
	Let $S=(S(t))_{t \geq 0}$ be an immediately differentiable semigroup on the Banach space $X$, then $S(t)$ is an immediately norm-continuous semigroup (see \cite{EngelNagel}, Definition 4.17 page 112).
\end{theorem}
\begin{proof}
If $S(t)$ is immediately differentiable then $S(t)$ is immediately differentiable with the uniform norm of $\mathcal{L}(X)$, for any $t>0$. This implies that the semigroup is immediately norm-continuous.
\end{proof}

\begin{theorem}\label{diff2}
	If $\mathcal{A}$ is the generator of an immediately norm-continuous exponentially stable semigroup then
	$$
	\lim_{\lambda\rightarrow\pm\infty}\|(i\lambda I-\mathcal{A})^{-1}\|=0.
	$$
\end{theorem}
\begin{proof}
See \cite{EngelNagel},  Corollary 4.19 page 114.
\end{proof}

\begin{theorem} \label{diff3}
	The semigroup  $S=(S(t))_{t \geq 0}$ defined by system \eqref{eq1}--\eqref{ci} is not differentiable
\end{theorem}

\begin{proof}
To show that  the semigroup is not   differentiable we only need to prove that there exists a sequence $\omega_n$ of real numbers such that 
	\begin{align} \label{condi}
		\lim_{n\to \infty}\left\|(i\omega_nI-\mathcal{A})^{-1} \right\|>0. 
	\end{align} 	
	We now consider $G_n=(0,\phi_n,0)$,  where	$\phi_n$ are the unitary eigenfunctions of the Laplace operator with homogeneous Dirichlet conditions on the boundary of $\Omega$. That is $-\Delta \phi_n=\lambda_n\phi_n$. 
	Let $U_n=(u_n,v_n,\theta_n)\in D(\mathcal{A})$ the unique solution of the equation
	\begin{align*}
		\left( i\omega_n I-\mathcal{A}\right)U_n=G_n. 
	\end{align*}
		In terms of the components of the system we have 
	\begin{align*}
		i\omega_n u_n-v_n&=0\\
		i\rho \omega_n v_n+\Delta \left( c\Delta u_n-\eta \theta_n\right)&=\phi_n\\
		ia\omega_n\theta_n+\Delta(d\Delta \theta_n-b\theta_n+\eta v_n)&=0. 
	\end{align*}
		To solve the above system we look for the solutions of the form  
	\begin{align*}
		u_n=A_n \phi_n, \ \ v_n=i\omega_n A_n\phi_n,\ \ \theta_n= B_n\phi_n.
	\end{align*}
Substitution of $(u_n,v_n,\theta_n)$ into the above system yields 
		\begin{align*}
	i \rho \omega_n \left(i\omega_n A_n \phi_n \right) +\Delta (cA_n\Delta \phi_n-\eta B_n \phi_n)= \phi_n,\\
	i a \omega_n B_n\phi_n+\Delta (d\textcolor{red}{ B_n\Delta\phi_n}-bB_n\phi_n+\eta i \omega_nA_n\phi_n)=0.
	\end{align*}
The above system is equivalent to 
	\begin{align}
		A_n(c\lambda^2_n-\rho \omega_n^2)+B_n \eta \lambda_n=1,\label{Alg1}\\
		-i\eta \omega_n\lambda _nA_n+B_n(ia \omega_n+ d\lambda_n^2 +b\lambda_n)=0\label{Alg2}
	\end{align}
	where $\lambda_n$ are eigenvalues corresponding to the eigenfunctions $\phi_n$. We recall that $\lambda_n\to \infty \ \ (\text{as} \ n\to \infty)$. Taking 
	$$\omega_n=\pm\sqrt{\dfrac{c}{\rho}}\lambda_n.$$
	We have that 
	$$B_n=(\eta \lambda_n)^{-1}.$$
Substitution of $B_n$ into \eqref{Alg2} yields  
	\begin{align*}
		A_n=\pm\dfrac{ia \omega_n+d\lambda_n^2+b\lambda_n}{i\eta^2 \lambda_n^3 \sqrt{\dfrac{c}{\rho }}}.
	\end{align*}
	Hence we have that 
	$$
	U_n=(u_n,v_n,\theta_n)=(A_n \phi_n, i\omega_n A_n\phi_n,  B_n\phi_n).
	$$
Therefore,
	$$
	\|U_n\|_{\mathcal{H}}^2\geq \|v_n\|_{L^2}^2 \approx  \dfrac{a^2 \omega_n^2+(d\lambda_n^2+b\lambda_n)^2}{\eta^4 \lambda_n^6 {\dfrac{c}{\rho }}}\omega_n^2\rightarrow \frac{d^2}{\eta^4}>0.
	$$
Since $U_n=(i\omega_n-\mathcal{A})^{-1}G_n$
	we see that  condition \eqref{condi} holds.
	\end{proof}

	\begin{remark}
We also note that the semigroup is not analytic. This is a bit surprising when compared to the results obtained at \cite{Z1}.
		\end{remark}
		\section{Impossibility of localization } 
		
		Although the semigroup is not analytic, we can employ alternative arguments to demonstrate that the only solution that is identically zero on a set of non-zero measure is the trivial solution. To this end, it is sufficient to prove  the uniqueness of solutions for the backward-in-time system
		\begin{align*}
			\rho u_{tt}=-c \Delta ^2 u+\eta \Delta \theta,\\
			a\theta_t=d\Delta^2 \theta-b\Delta \theta-\eta\Delta u_t.
		\end{align*}
		with homogeneous Dirichlet boundary conditions \eqref{fr1}.
		
				To show the uniqueness of the solutions to this problem it is enough to show that the only solution for the problem with the  null initial conditions
		\begin{align*}
			u(x,0)=u_t(x,0)=\theta(x,0)=0, \ \ \forall x\in \Omega.
		\end{align*}
		As usual, we will use the argument from the Lagrange identities.
		First, we consider the functions  
		\begin{align*}
		\mathscr{L}_1(t)=\dfrac{1}{2}\int_{\Omega}\left( \rho \left|u_t \right|^2+c\left| \Delta u \right|^2+a \left|\theta \right|^2   \right) \ d\Omega,
		\end{align*}
				and 
				\begin{align*}
		\mathscr{L}_2(t)=\dfrac{1}{2}\int_{\Omega}\left( \rho \left|u_t \right|^2+c\left| \Delta u \right|^2-a \left|\theta \right|^2   \right) \ d\Omega.
				\end{align*}
								We get 
								\begin{align*}
			\textperiodcentered\dot{\mathscr{L}}_1(t)=\int_{\Omega}\left( b\left|\nabla \theta \right|^2+d\left|\Delta \theta \right|^2   \right) \  d\Omega,
			\end{align*}
			and 
					\begin{align*}
				\dot{\mathscr{L}}_2(t)=-\int_{\Omega}\left( b\left|\nabla \theta \right|^2+d\left|\Delta \theta \right|^2   \right) \  d\Omega +2\eta \int_{\Omega} u_t\Delta \theta \ d\Omega.
			\end{align*}
	In order to obtain an alternating expression to the definition of the function $\mathscr{L}_2(t)$ we consider the following relations 
	\begin{align*}
		\int_{0}^{t}\int_{\Omega} \rho u_{tt}(s)u_t(2t-s) \ d\Omega ds+\int_{0}^{t}\int_{\Omega} c\Delta u(s)\Delta u_t(2t-s)\ d\Omega ds&=\int_{0}^{t}\int_{\Omega}\eta \Delta \theta (s) u_t(2t-s) \ d\Omega ds,\\ \\
		\int_{0}^{t}\int_{\Omega} \rho u_{tt}(2t-s)u_t(s) \ d\Omega ds+\int_{0}^{t}\int_{\Omega} c\Delta u(2t-s)\Delta u_t(s)\ d\Omega ds&=\int_{0}^{t}\int_{\Omega}\eta \Delta \theta (2t-s) u_t(s) \ d\Omega ds,
			\end{align*}	
				\begin{align*}	
	\int_{0}^{t}\int_{\Omega} a\theta _t(s)\theta(2t-s) \ d\Omega ds&=\int_{0}^{t}\int_{\Omega} d\Delta \theta(s)\Delta \theta(2t-s) \ d\Omega ds +\int_{0}^{t}\int_{\Omega} b\nabla \theta(s)\nabla \theta(2t-s) \ d\Omega ds\\&-\int_{0}^{t}\int_{\Omega} \eta \Delta u_t(s)\theta (2t-s) \ d\Omega ds,
	\end{align*}	
	and
	\begin{align*}
			\int_{0}^{t}\int_{\Omega} a\theta _t(2t-s)\theta(s) \ d\Omega ds&=\int_{0}^{t}\int_{\Omega} d\Delta \theta(2t-s)\Delta \theta(s) \ d\Omega ds +\int_{0}^{t}\int_{\Omega} b\nabla \theta(2t-s)\nabla \theta(s) \ d\Omega ds\\&-\int_{0}^{t}\int_{\Omega} \eta \Delta u_t(2t-s)\theta (s) \ d\Omega ds.
	\end{align*}	
	By combining \textcolor{red}{these} equalities with alternative signs, integrating with respect to time, and taking into account the initial conditions, we find the following relation:
	\begin{align*}
		\int_{\Omega}(a\left| \theta\right| ^2+c\left|\Delta u \right|^2 ) \ d\Omega= \int_{\Omega} \rho \left|u_t \right| ^2 \ d\Omega.
	\end{align*} 
	If we go to the definition of the function $\mathscr{L}_2(t)$ and consider this previous equality we find
	\begin{align*}
		\mathscr{L}_2(t)=\int_{\Omega} c \left|\Delta u \right| ^2 \ d\Omega.
	\end{align*}
			Therefore if we \textcolor{red}{select} $\epsilon<1$ and consider 
			\begin{align*}
				\mathscr{L}(t)=	\mathscr{L}_2(t)+\epsilon 	\mathscr{L}_1(t),
			\end{align*}
			we see that
			\begin{align*}
				\dot{\mathscr{L}}(t)&=-(1-\epsilon)\int_{\Omega}\left(b \left| \nabla \theta\right|^2+ d\left| \Delta \theta\right| ^2\right) \ d\Omega  + 2\eta \int_{\Omega} u_t \Delta \theta \ d\Omega\\
				&\leq k \int_{\Omega} \left|u_t \right| ^2 \ d\Omega, \ \  k>0 ,
			\end{align*}
			after the use of the Holder inequality.			As $\epsilon>0$ we can also see that
						\begin{align*}
				\dot{\mathscr{L}}(t) \leq k^* \mathscr{L}(t),
			\end{align*}
			where $k^*$ is a calculable constant.
			It then follows that
			$$\mathscr{L}(t)\leq \mathscr{L}(0)e^{k^*t}=0,$$
			as $\mathscr{L}(0)=0.$
			Therefore, we obtain that $\mathscr{L}(t)=0$ for every $t\geq 0$ and we can conclude:
			\begin{theorem}
			Let $(u,\alpha)$ be a solution to the problem determined for the backward in time system with null initial conditions. Then
			$(u,\theta)=(0,0),$
			for every $t\geq 0.$
			\end{theorem}
				\section{Uniqueness and Instability }

		In this section, we analyze the system defined by equations \eqref{eq1} and \eqref{eq2} in the case where $c < 0$. In this scenario, the problem is not expected to be well-posed in the Hadamard sense due to its instability. Nevertheless, we prove the uniqueness of solutions.

To study the problem, it is convenient to integrate equation \eqref{eq2} with respect to time, yielding:
		\begin{align}
			a\theta=b\Delta \alpha-d\Delta^2\alpha-\eta  \Delta u+a\theta(0)+\eta \Delta u(0),\label{eqq1}
		\end{align}
		where 
		\begin{align*}
			\alpha(x,t)=\int_{0}^{t}\theta(x,s) \ ds.
		\end{align*}
			We note that if $\Phi (x)$ is the solution to the equation
	\begin{align*}
		b\Delta \Phi-d\Delta^2 \Phi=a\theta (0)+\eta \Delta u(0),
	\end{align*}
	with null Dirichlet boundary conditions on $\Omega,$ we can write the equation \eqref{eqq1} as
	\begin{align*}
		a\theta=b\Delta\Psi-d\Delta^2\Psi-\eta \Delta u,
	\end{align*}
	where $\Psi=\alpha+\Phi$.
	We define the function
	\begin{align*}
		\mathscr{F}(t)=\int_{\Omega} \rho u^2 \ d \Omega+\int_{0}^{t}\int_{\Omega}\left(b\left| \nabla \Psi\right|^2+d\left|\Delta \Psi \right|^2   \right) \ d\Omega ds+ \omega^2 (t+t_0).
	\end{align*}
	
	Here $\omega $ and $t_0$ are two nonnegative constants that will be selected later. We have
	\begin{align*}
		\dot{\mathscr{F}}(t)=2\int_{\Omega} \rho u u_t \ d\Omega+2\int_{0}^{t}\int_{\Omega} \left(b\nabla \Psi\nabla \theta+d\Delta \Psi \Delta \theta \right) \ d \Omega \ d s+\int_{\Omega}\left( b\left|\nabla \Phi \right|^2+d \left|\Delta \Phi \right|^2 \right)  d\Omega+2\omega(t+t_0),
	\end{align*}
	and
		\begin{align*}
		\ddot{\mathscr{F}}(t)=2\int_{\Omega} \left( \rho \left| u_t\right| ^2+\rho u u_{tt} \right) d\Omega+2\int_{\Omega} \left(b\nabla \Psi\nabla \theta+d\Delta \Psi \Delta \theta \right) \ d \Omega +2\omega(t+t_0).
	\end{align*}
	We note that
		\begin{align*}
	\int_{\Omega} \left(  \rho u u_{tt} +b\nabla \Psi\nabla \theta+d\Delta \Psi \Delta \theta \right) \ d \Omega&=-\int_{\Omega} \left(c\left|\Delta u \right| ^2+a\left|\theta \right|^2  \right) d\Omega \\
	&=\int_{\Omega} \rho \left|u_t \right| ^2 d\Omega+2\int_{0}^{t}\int_{\Omega}\left( b\left| \nabla \theta\right|^2+d\left| \Delta \theta\right|^2  \right) d\Omega ds-E(0),
	\end{align*}
	where the second equality follows from \eqref{eq1.5}. We obtain
	\begin{align*}
		\ddot{\mathscr{F}}(t)=4\int_{\Omega}\rho \left| u_t\right|^2 d\Omega+4\int_{0}^{t}\int_{\Omega}\left( b\left| \nabla \theta\right|^2+d\left|\Delta \theta \right| ^2 \right)  d\Omega d s-2\left( E(0)-\omega\right) .
			\end{align*}
	The use of the Schwarz inequality implies that
	\begin{align}
	\ddot{\mathscr{F}}(t) \mathscr{F}(t)-\left( \dot{\mathscr{F}}(t)-\nu\right)^2 \geq -2\left(\omega+E(0) \right) \mathscr{F}(t)\label{eqqqq}
	\end{align}
	where
	\begin{align*}
		\nu=\int_{\Omega}\left(b\left| \nabla \Phi \right|^2+d\left| \Delta \Phi \right|^2   \right) d\Omega. 
	\end{align*}
	Now, we can obtain the uniqueness of the solutions from \eqref{eqqqq}. We know that to prove the uniqueness it is sufficient to show that the only solution to null initial conditions is the null solution. In this case we have that $E(0)=\nu=0$ and if we take $\omega=0$ we get
	$$	\ddot{\mathscr{F}}(t) \mathscr{F}(t)\geq (\dot{\mathscr{F}}(t))^2$$
	This inequality brings to the estimate. See \cite{Ames, flavin}
	$$\mathscr{F}(t)\leq \mathscr{F}(0)^{1-t/t_1}\mathscr{F}(t_1)^{t/t_1}, \ \ 0\leq t\leq t_1.$$
	
	As $\mathscr{F}(0)=0$ we conclude that $\mathscr{F}(t)=0$ for $ 0\leq t\leq t_1$ which shows that $u=\theta =0$ for every $ 0\leq t\leq t_1$ . And the uniqueness result is obtained.\\
	To prove the instability of the solutions we use \eqref{eqqqq} and select $\omega=-E(0)$. We also select $t_0$ large enough to guarantee that $\dot{\mathscr{F}}(t)> 2\nu$.
	 We see
	 	 \begin{align*}
	 	\mathscr{F}(t)\geq \dfrac{\dot{\mathscr{F}}(0)\mathscr{F}(0)}{\dot{\mathscr{F}}(0)-2\nu} \exp \left( \dfrac{\dot{\mathscr{F}}(0)-2\nu}{\mathscr{F}(0)}t\right) -2\nu \dfrac{	\mathscr{F}(0)}{\dot{\mathscr{F}}(0)-2\nu}
	 \end{align*}
	 which guarantee the exponential instability.
	 A similar estimate can be obtained when $E(0)=0$ but $\dot{\mathscr{F}}(0)>0.$
	 
	 \begin{theorem}
	 For the problem determined by the system \eqref{eq1},\eqref{eq2}  with the homogeneous Dirichlet boundary conditions in the case $c<0$, we have
	 \begin{enumerate}
	 \item There exists at most one solution.
	 \item When $E(0)<0$ or $E(0)=0$ but $\dot{\mathscr{F}}(0)>0$ the solution is exponentially unstable.
	 \end{enumerate}
	 \end{theorem}
 
 	\section{Quasi-Static Case }
 	 We continue assuming that $c<0$, but we restrict our attention to the one-dimensional quasi-static case.
 	 Our system becomes
 	 \begin{align*}
 	 	cu_{xxxx}&=\eta \theta_{xx},\\
 	 	a\theta_t&=b\theta_{xx}-d\theta_ {xxxx}-\eta u_{xxt}.
 	 	 \end{align*}
 	 
 	We study this system in the interval $[0,1]$ with the boundary conditions
 	\begin{align*}
 	u(0,t)=u_{xx}(0,t)=u(1,t)=u_{xx}(1,t)=0,\\
 	\theta (0,t)=\theta_{xx}(0,t)=\theta (1,t)=\theta_{xx}(1,t)=0,\ \ t\geq 0.
 	\end{align*}
 	and the initial condition
 	$$\theta(x,0)=\theta^0(x).$$ 
 	We can integrate the first equation with respect to the spatial variable to obtain
 	\begin{align*}
 		c(u_{xxx}-u_{xxx}(0))=\eta (\theta _{x}-\theta_x(0)),
 	\end{align*}
 	and 
 	\begin{align*}
 		c(u_{xx}-xu_{xxx}(0))=\eta (\theta-x\theta_x(0)).
 	\end{align*}
 	At the point $x=1$, we have $u_{xx}(1)=\theta(1)=0.$ Therefore $cu_{xxx}(0)=\eta \theta_x(0)$ and we get 
 	$$cu_{xx}=\eta \theta.$$
 	After a time derivation we see
 	\begin{align*}
 		\eta u_{xxt}=\dfrac{\eta^2}{c} \theta_t.
 	\end{align*}
 	Going back to the second equation of this section, we find
 	\begin{align}
 		\left(a+\dfrac{\eta^2}{c}\right) \theta_t =b \theta_{xx}-d\theta _{xxxx}. \label{eqqq2}
 	\end{align}
 	In the case that $a>-\dfrac{\eta^2}{c}$ the study of equation \eqref{eqqq2} is well known. We can obtain the existence of an analytic semigroup that provides solutions and the exponential stability. At the same time we can guarantee that if $\theta_t\in L^2$ then $\theta_{xx}\in L^2$ for every $t>0$. Then, going back to the first equation of this section we can obtain $u\in H^2$. To provide an estimate for the behavior of $u(x,t)$ we see
 	\begin{align*}
 		-c\int_{0}^{1} \left| u_{xx}\right|^2 dx&= -\eta \int_{0}^{1} \theta u_{xx}\leq k\left( \int_{0}^{1}\left| \theta\right|^2 dx \right) ^{1/2}\left( \int_{0}^{1}\left| u_{xx}\right|^2 dx \right)^{1/2}. 
 	\end{align*}
 	Then we see
 	\begin{align*}
 		-c \int_{0}^{1}\left| u_{xx}\right|^2 \ dx \leq k^* \left( \int_{0}^{1}\left|\theta \right|^2 \ dx \right)  \leq k^* \left(\int_{0}^{1} \left| \theta^0\right| ^2 dx\right) \exp (-\omega t),
 	\end{align*}
 	which gives the exponential decay of $u$ in the $H^2$-norm.
 	\begin{theorem}
 	The problem determined by the one-dimensional quasi-static solutions in the case that $a>- \dfrac{\eta^2}{c}$ satisfies
 	\begin{enumerate}
 	\item There are solutions for every initial data in $L^2$, and these solutions are analytic with respect to time.
 	\item The solutions decay exponentially.
 	\end{enumerate}
 	\end{theorem}

\section{Conclusion}
In this article, we have analyzed the system of equations that governs the thermoelastic deformation of a plate, where heat conduction is modeled by the Green-Naghdi Type I theory with higher-order spatial derivatives. Assuming standard conditions on the constitutive coefficients, we have established the following qualitative properties:

\begin{enumerate}
    \item Existence of a semigroup that defines the solutions in an appropriate Hilbert space. Uniqueness of solutions is also satisfied.
    \item Exponential stability of the solutions.
    \item Non-differentiability of the semigroup, implying non-analyticity.
    \item Impossibility of localizing the solutions.
\end{enumerate}
Later, we consider the case when the elastic parameter is negative and we show:
\begin{enumerate}
    \item Uniqueness and instability of solutions.
    \item For the one-dimensional quasi-static problem, existence and exponential stability of solutions, provided specific parameter conditions are satisfied.
\end{enumerate}

It is instructive to compare these results with those for the classical case, where higher-order derivatives in heat conduction are absent. Properties 1, 2, and 4 align closely with the classical setting; however, property 3 marks a significant difference. Notably, the introduction of stronger dissipation unexpectedly leads to a loss of solution regularity, a phenomenon that merits further attention. Properties 5 and 6 are novel, even in the absence of higher-order derivatives.			
			
\section*{Acknowledgments} 

Jaime Muñoz Rivera was supported by CNPq project 307947/2022-0  and Fondecyt project 1230914.

\section*{Conflict of interest}	
This work does not have any conflicts of interest.

	\end{document}